\documentclass[11pt,letterpaper]{amsart}
\usepackage{amsmath,amssymb,txfonts}
\usepackage{amssymb}
\usepackage{amsxtra}
\usepackage{amsthm, color}
\usepackage{txfonts}
\usepackage{graphicx}
\usepackage{times}
\usepackage{citeref}
\usepackage{tikz}
\usepackage{pgfplots}
\usepackage{tikz-3dplot}
\numberwithin{equation}{section}
\usepackage{pb-diagram}
\usepackage{tkz-euclide}
\newcommand{\hgline}[2]{
	\pgfmathsetmacro{\thetaone}{#1}
	\pgfmathsetmacro{\thetatwo}{#2}
	\pgfmathsetmacro{\theta}{(\thetaone+\thetatwo)/2}
	\pgfmathsetmacro{\phi}{abs(\thetaone-\thetatwo)/2}
	\pgfmathsetmacro{\close}{less(abs(\phi-90),0.0001)}
	\ifdim \close pt = 11pt
	\draw[red] (\theta+180:1) -- (\theta:1);
	\else
	\pgfmathsetmacro{\R}{tan(\phi)}
	\pgfmathsetmacro{\distance}{sqrt(1+\R^2)}
	\draw[red] (\theta:\distance) circle (\R);
	\fi
}

\newtheorem{prop}{Proposition}[section]
\newtheorem{theorem}[prop]{Theorem}

\newtheorem{remark}[prop]{Remark}

\usepackage{tikz}
\usetikzlibrary{calc}

\makeatletter

\def\hyper@x#1,#2\relax{#1}
\def\hyper@y#1,#2\relax{#2}
\def\hyper@coords#1{#1}

\newif\ifhyper@vertical

\def\hyper@computer#1#2{%
	\edef\hyper@toscan{(#1)}
	\tikz@scan@one@point\hyper@coords\hyper@toscan
	\edef\hyper@sx{\the\pgf@x}
	\edef\hyper@sy{\the\pgf@y}
	\edef\hyper@toscan{(#2)}
	\tikz@scan@one@point\hyper@coords\hyper@toscan
	\edef\hyper@ex{\the\pgf@x}
	\edef\hyper@ey{\the\pgf@y}
	\pgfmathsetmacro{\hyper@mx}{(\hyper@ex + \hyper@sx)/2}
	\pgfmathsetmacro{\hyper@my}{(\hyper@ey + \hyper@sy)/2}
	\pgfmathsetmacro{\hyper@dx}{\hyper@ex - \hyper@sx}
	\pgfmathparse{\hyper@dx == 0 ? "\noexpand\hyper@verticaltrue" : "\noexpand\hyper@verticalfalse"}
	\pgfmathresult
	\ifhyper@vertical
	\edef\hyper@cmd{-- (\tikztotarget)}
	\else
	\pgfmathsetmacro{\hyper@dy}{\hyper@ey - \hyper@sy}
	\pgfmathsetmacro{\hyper@t}{\hyper@my/\hyper@dx}
	\pgfmathsetmacro{\hyper@cx}{\hyper@mx + \hyper@t * \hyper@dy}
	\pgfmathsetmacro{\hyper@radius}{veclen(\hyper@cx - \hyper@sx, \hyper@sy)}
	\pgfmathsetmacro{\hyper@sangle}{180 - atan2(\hyper@sy,\hyper@cx-\hyper@sx)}
	\pgfmathsetmacro{\hyper@eangle}{180 - atan2(\hyper@ey,\hyper@cx-\hyper@ex)}
	\edef\hyper@cmd{arc[radius=\hyper@radius pt, start angle=\hyper@sangle, end angle=\hyper@eangle]}
	\fi
}

\def\hyper@disc@computer#1#2{%
	\edef\hyper@toscan{(#1)}
	\tikz@scan@one@point\hyper@coords\hyper@toscan
	\edef\hyper@sx{\the\pgf@x}
	\edef\hyper@sy{\the\pgf@y}
	\edef\hyper@toscan{(#2)}
	\tikz@scan@one@point\hyper@coords\hyper@toscan
	\edef\hyper@ex{\the\pgf@x}
	\edef\hyper@ey{\the\pgf@y}
	\pgfmathsetmacro{\hyper@det}{\hyper@sx * \hyper@ey - \hyper@sy * \hyper@ex}
	\pgfmathparse{\hyper@det == 0 ? "\noexpand\hyper@verticaltrue" : "\noexpand\hyper@verticalfalse"}
	\pgfmathresult
	\ifhyper@vertical
	\edef\hyper@cmd{-- (\tikztotarget)}
	\else
	\pgfmathsetmacro{\hyper@mx}{(\hyper@ex + \hyper@sx)/2}
	\pgfmathsetmacro{\hyper@my}{(\hyper@ey + \hyper@sy)/2}
	\pgfmathsetmacro{\hyper@dx}{\hyper@ex - \hyper@sx}
	\pgfmathsetmacro{\hyper@dy}{\hyper@ey - \hyper@sy}
	\pgfmathsetmacro{\hyper@dradius}{\pgfkeysvalueof{/tikz/hyperbolic disc radius}}
	\pgfmathsetmacro{\hyper@t}{((\hyper@dradius)^2 - \hyper@sx * \hyper@ex - \hyper@sy * \hyper@ey)/(2 * (\hyper@sx * \hyper@ey - \hyper@sy * \hyper@ex))}
	\pgfmathsetmacro{\hyper@radius}{sqrt((\hyper@t)^2 + .25) * veclen(\hyper@dx,\hyper@dy)}
	\pgfmathsetmacro{\hyper@cx}{\hyper@mx + \hyper@t * \hyper@dy}
	\pgfmathsetmacro{\hyper@cy}{\hyper@my - \hyper@t * \hyper@dx}
	\pgfmathsetmacro{\hyper@sangle}{atan2(\hyper@sy-\hyper@cy,\hyper@sx - \hyper@cx)}
	\pgfmathsetmacro{\hyper@eangle}{atan2(\hyper@ey-\hyper@cy,\hyper@ex - \hyper@cx)}
	\pgfmathsetmacro{\hyper@eangle}{\hyper@eangle > \hyper@sangle + 180 ? \hyper@eangle - 360 : \hyper@eangle}
	\edef\hyper@cmd{arc[radius=\hyper@radius pt, start angle=\hyper@sangle, end angle=\hyper@eangle]}
	\fi
}

\def\hyper@plane@tangent#1#2{%
	\edef\hyper@toscan{(#1)}
	\tikz@scan@one@point\hyper@coords\hyper@toscan
	\edef\hyper@sx{\the\pgf@x}
	\edef\hyper@sy{\the\pgf@y}
	\edef\hyper@toscan{(#2)}
	\tikz@scan@one@point\hyper@coords\hyper@toscan
	\edef\hyper@ex{\the\pgf@x}
	\edef\hyper@ey{\the\pgf@y}
	\pgfmathsetmacro{\hyper@ex}{\hyper@ex - \hyper@sx}
	\pgfmathsetmacro{\hyper@ey}{\hyper@ey - \hyper@sy}
	\pgfmathparse{\hyper@ex == 0 ? "\noexpand\hyper@verticaltrue" : "\noexpand\hyper@verticalfalse"}
	\pgfmathresult
	\ifhyper@vertical
	\pgfmathsetmacro{\hyper@d}{\hyper@ey/1cm}
	\pgfmathsetmacro{\hyper@radius}{\hyper@sy * exp(\hyper@d) - \hyper@sy}
	\edef\hyper@cmd{-- ++(0,\hyper@radius pt)}
	\else
	\pgfmathsetmacro{\hyper@d}{\hyper@ex > 0 ? veclen(\hyper@ex,\hyper@ey) : -veclen(\hyper@ex,\hyper@ey)}
	\pgfmathsetmacro{\hyper@radius}{abs(\hyper@sy * \hyper@d / \hyper@ex)}
	\pgfmathsetmacro{\hyper@sangle}{90 + atan(\hyper@ey/\hyper@ex)}
	\pgfkeysgetvalue{/tikz/hyperbolic plane target angle}{\hyper@eangle}
	\ifx\hyper@eangle\pgfutil@empty
	\pgfmathsetmacro{\hyper@d}{\hyper@d/1cm}
	\pgfmathsetmacro{\hyper@ey}{\hyper@ey/1cm}
	\pgfmathsetmacro{\hyper@tanhd}{tanh(\hyper@d)}
	\pgfmathsetmacro{\hyper@eangle}{acos((\hyper@d * \hyper@tanhd - \hyper@ey)/(\hyper@d - \hyper@ey * \hyper@tanhd))}
	\fi
	\edef\hyper@cmd{arc[radius=\hyper@radius pt, start angle=\hyper@sangle, end angle=\hyper@eangle]}
	\fi
}

\tikzset{%
	hyperbolic disc radius/.initial={1cm},
	hyperbolic plane/.style={
		to path={
			\pgfextra{\hyper@computer\tikztostart\tikztotarget}
			\hyper@cmd
		}
	},
	hyperbolic plane tangent/.style={
		to path={
			\pgfextra{\hyper@plane@tangent\tikztostart\tikztotarget}
			\hyper@cmd
		}
	},
	hyperbolic disc/.style={
		to path={
			\pgfextra{\hyper@disc@computer\tikztostart\tikztotarget}
			\hyper@cmd
		}
	},
	hyperbolic plane target angle/.initial={},
}

\makeatother
\begin{document}
	
	\title[\tiny Essential $p$-capacity-volume estimates for rotationally symmetric manifolds]{Essential $p$-capacity-volume estimates for\\
		 rotationally symmetric manifolds}
	\author{\tiny Xiaoshang Jin}
	\address{School of Mathematics and Statistics, Huazhong University of Science and Technology, Wuhan, Hubei 430074, China}
	\email{jinxs@hust.edu.cn}
	\author{\tiny Jie Xiao}
	\address{Department of Mathematics \& Statistics,
		Memorial University, St. John's, NL A1C 5S7, Canada}
	\email{jxiao@math.mun.ca}
	
	\thanks{The 1st-named \& 2nd-named authors were supported by {NNSF of China
			\# 12201225} and NSERC of Canada \# 202979 respectively.}
	
	\date{}

\date{}


\maketitle


\begin{abstract} Given $p\in [1,\infty]$, this article presents the novel basic volumetric estimates for the relative $p$-capacities with their applications to finding not only the sharp weak $(p,q)$-imbeddings but also the precise lower bounds of the principal $p$-frequencies, which principally live in the rotationally symmetric manifolds.
\end{abstract}

\tableofcontents

\section*{Introduction}\label{s0}

Given $p\in (0,\infty]$, if
$$\nu_n=|\mathbb B^n|\ \ \&\ \ n\nu_n=|\mathbb S^{n-1}|
$$
are respectively the volume of the unit open ball $\mathbb B^n$ \& the surface area of the unit closed sphere $\mathbb S^{n-1}=\partial\mathbb B^n$ within the $1\le n$-dimensional Euclidean space $\mathbb R^n$, then it is well-known that for any compact subset $K$ of a bounded open subdomain $\Omega$ of an $n$-dimensional complete Riemannian manifold $(\mathbb M^n,g)$ with volume element $d\upsilon_g$, the relative functional/variational $p$-capacity
$$
{\rm cap}_p(K,\Omega)=\begin{cases} \inf\Big\{\int_{\Omega}|\nabla f|^p\,d\upsilon_g:\ \mathrm{1}_K\le f\in C_0^1(\Omega)\Big\}&\text{as}\ \ p\in (0,\infty);\\
\inf\Big\{\underset{x\in \Omega}{\sup}|\nabla f(x)|:\ \mathrm{1}_K\le f\in C_0^1(\Omega)\Big\}&\text{as}\ \ p=\infty,
\end{cases}
$$
defines the minimum energy of a condenser $(K,\Omega)$ (cf. \cite{PS2} for $p=2=n-1$) whose importance in mathematical physics can be three-dimensionally pictured as:

\begin{center}
	\includegraphics[scale=0.166,keepaspectratio]{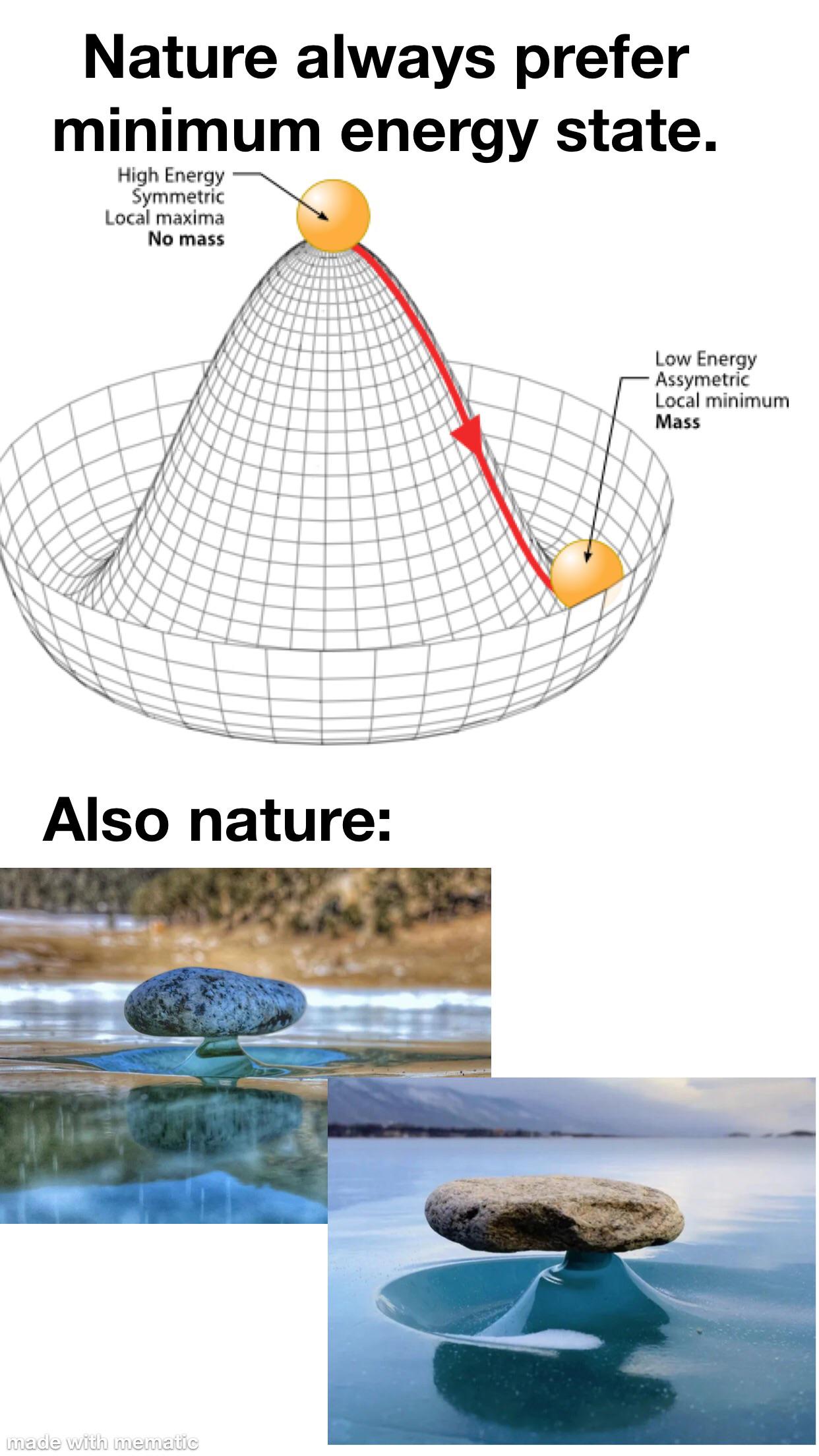}\\
 {\tiny https://www.reddit.com/r/physicsmemes/comments/tcbosm/principle$_{-}$of$_{-}$minimum-energy/?rdt=55327}
\end{center}

Referring to not only \cite{Sa} for the Euclidean space $\mathbb R^n$ but also \cite{JX} for the hyperbolic space $\mathbb H^{n+1}$, we have
$$
{\rm cap}_p(K,\Omega)=0\ \ \forall\ \ p\in (0,1).
$$
Consequently, our main interest in these capacities then lies with using the volume pair $\{|K|, |\Omega|\}$ to find the essential lower bound of each relative capacity
 $$
 {\rm cap}_p(K,\Omega)\ \ \forall\ \ p\in [1,\infty].
 $$

Note that if
$$
\begin{cases}
d\upsilon_g(x)=dx=\text{the Lebesgue volume element}\ \ \forall\ \ x\in \mathbb M^n=\mathbb R^n;\\
rK=\{rx\in\mathbb{R}^n:x\in K\};\\
\emptyset\not=K^\circ=\text{the interior of $K$};\\
f\in C^1(K^\circ);\\
f_r(\cdot)=f(\frac{\cdot}{r}):\ \ rK\rightarrow\mathbb{R},
\end{cases}
$$
then
$$
\begin{cases}
|K|^{\frac{p}{n}-1}\int_K |\nabla f(x)|^pdx=|rK|^{\frac{p}{n}-1}\int_{rK} |\nabla f_r(y)|^pdy&\ \ \text{as}\ \ p\in [1,\infty);\\
|K|^{\frac{1}{n}}\underset{x\in K}{\sup}|\nabla f(x)|=|rK|^{\frac{1}{n}}\underset{y\in rK}{\sup}|\nabla f_r(y)|&\ \ \text{as}\ \ p=\infty.
\end{cases}
$$
In other words, $$\int_K \left(\frac{|\nabla f(x)|^p}{|K|^{1-\frac{p}{n}}}\right)\,dx\ \ \&\ \ \sup_{x\in K}\left(\frac{|\nabla f(x)|}{|K|^{-\frac1n}}\right)$$
are scaling invariant. Thus, within  the 1st two sections of this paper we are naturally led to utilize the powers
$$
1-\frac{p}{n}\le 1-\frac1n
$$
for a brand-new investigation of the precise-to-optimal lower bound of $\mathrm{cap}_p(K,\Omega)$ based on such a power of the volume-rate as $\big(|K|/|\Omega|\big)^\alpha$, see also:
$$
\begin{cases}\text{Theorem \ref{t11} when $\alpha<1-\frac{p}{n}$ lives in a complete Riemannian manifold};\\
\text{Theorem \ref{t21} when $\alpha\ge 1-\frac{p}{n}$ lives in a rotationally symmetric manifold}.
\end{cases}
$$
Furthermore, within the 2nd two sections of this paper, we employ Theorem \ref{t21} to not only sharply embed the Sobolev $p$-space
$$
W^{1,p}_0(\Omega)=\overline{\left\{f\in C^1_0(\Omega):\ \|\nabla f\|_p=\left(\int_{\Omega}|\nabla f|^p\,d\upsilon_g\right)^{\frac{1}{p}}<\infty\right\}}
$$
into the weak Lebesgue $q$-space
$$
L^{q,\infty}(\Omega)=\begin{cases} \left\{f:\ \|f\|_{q,\infty}=\sup\limits_{\lambda>0}\lambda|\{x\in\Omega:\ |f(x)|\geq \lambda\}|^{\frac{1}{q}}<\infty\right\}&\text{as}\ \ q\in (0,\infty);\\
\left\{f:\  \|f\|_{\infty}=\underset{x\in \Omega}{\text{esssup}}\,|f(x)|<\infty\right\}&\text{as}\ \ q= \infty,
\end{cases}
$$
but also precisely handle the volumetric behavior of the principal $p$-frequency (cf. \cite[p. 24]{PS} for $p=2=n$)
$$
\lambda_{1,p}(\Omega)=\begin{cases}\inf\Bigg\{\frac{\int_{\Omega}|\nabla f|^p\,d\upsilon_g}{\int_{\Omega}|f|^p\,d\upsilon_g}:\ f\in C_0^1(\Omega)\Bigg\}&\text{as}\ \ p\in [1,\infty);\\
\inf\Bigg\{\frac{\sup\limits_{\Omega}|\nabla f|}{\sup\limits_{\Omega}|f|}:\ f\in C_0^1(\Omega)\Bigg\}&\text{as}\ \ p=\infty,
\end{cases}
$$
see also:
$$
\begin{cases}
\text{Theorem \ref{t31} when $\alpha=1-\frac{p}{n}$ lives in a rotationally symmetric manifold;}\\
\text{Theorem \ref{t41} when $\alpha=1>1-\frac{p}{n}$
	 lives in a rotationally symmetric manifold.}
 \end{cases}
$$

\section{Vanishing relative $p$-capacities under $\alpha<1-\frac{p}{n}$}\label{s1}

\begin{theorem}\label{t11}
Given
$$
\alpha< 1-\frac{p}{n}\le 1-\frac1n,
$$
suppose that $(\mathbb M^n,g)$ is an $n$-dimensional complete Riemannian manifold. There hold two vanishing results for ${\rm cap}_p(K,\Omega)$.
\begin{itemize}
\item[\rm(i)] If $p\not=n$, then
\begin{equation}\label{1.1}
  0=\inf\limits_{K\subset\subset \Omega}\begin{cases}|K|^{-\alpha}{\rm cap}_p(K,\Omega)&\ \ \text{as}\ \ p\not=\infty;\\
   |K|^\frac1n{{\rm cap}_\infty(K,\Omega)}&\ \ \text{as}\ \ p=\infty.
   \end{cases}
 \end{equation}

\item[\rm(ii)] If $p=n,$ then
\begin{equation}\label{1.2}
0=\inf\limits_{K\subset\subset \Omega}|K|^{-1}{\exp\left(-\beta \big({\rm cap}_n(K,\Omega)\big)^{\frac{1}{1-n}}\right)}\ \ \forall\ \ \beta>n^{\frac{n}{n-1}}\nu_n^{\frac{1}{n-1}}.
\end{equation}

\end{itemize}
\end{theorem}

\begin{proof} (i)
	For $p\neq n,$ we select a point $x\in\Omega$ such that
	 $$
	 B(x,r)\subseteq B(x,R)\subseteq\Omega\ \ \text{for some small $r<R$},
	 $$
	 where $B(x,r)$ is $x$-centered geodesic open ball with radius $r$, thereby estimating
	$$
	\frac{{\rm cap}_p\big(\overline{B(x,r)},\Omega\big)}{\big|\overline{B(x,r)}\big|^\alpha}
	\leq \frac{{\rm cap}_p(\overline{B(x,r)},B(x,R))}{\big|\overline{B(x,r)}\big|^\alpha}\leq\frac{\left(\int_r^R S_t^{\frac{1}{1-p}}dt\right)^{1-p}}{\big|\overline{B(x,r)}\big|^\alpha},
	$$
	where
	$$
	\begin{cases}
	S_t=\big|\{y\in\Omega:{\rm dist}(y,x)=t\}\big|\sim n\nu_nt^{n-1}&\text{as}\ \ t\to 0;\\
	\big|\overline{B(x,t)}\big|\sim \nu_n t^n&\text{as}\ \ t\rightarrow 0;\\
	\big|\partial\overline{B(x,t)}\big|\sim n\nu_n t^{n-1}\text{as}\ \ t\rightarrow 0.
	\end{cases}
	$$
	In the above \& below, $U\sim V$ means that either $U\, \&\, V$ or $U^{-1}\, \&\, V^{-1}$ are equivalent infinitesimals.

	\par If $p=1$, then \cite[p.107, Lemma]{Maz6} yields a constant $C_1$ such that
	\begin{align*}
	\lim_{r\to 0}\frac{{\rm cap}_1\big(\overline{B(x,r)},\Omega\big)}{\big|\overline{B(x,r)}\big|^\alpha}
	&\leq\lim_{r\to 0} \frac{|\partial\overline{B(x,r)}|}{\big|\overline{B(x,r)}\big|^\alpha}\\
	&\le C_1\lim_{r\to 0}r^{n-1-\alpha n}\\
	&=0\ \ \text{since}\ \ \alpha<1-n^{-1}.
	\end{align*}
	Thus, the argument for \eqref{1.1} is divided into three situations.
	
	\par If $1<p<n$, then there is a constant $C_p>0$ such that
	$$
	\lim\limits_{r\rightarrow 0}\frac{\left(\int_r^R S_t^{\frac{1}{1-p}}dt\right)^{1-p}}{\big|\overline{B(x,r)}\big|^\alpha}=C_p\lim\limits_{r\rightarrow 0}
	\frac{r^{n-1}}{r^{n\alpha}}=0\ \ \text{
	since}\ \ n\alpha<n-p<n-1.
$$
However, if $n<p<\infty,$ then $\alpha<1-\frac{p}{n}<0.$ We can still derive that the above limit is $0$ because not only the numerator is bounded but also the  denominator tends on infinity. Therefore, we always have	
	$$
	\lim\limits_{r\rightarrow 0}\frac{{\rm cap}_p\big(\overline{B(x,r)},\Omega\big)}{\big|\overline{B(x,r)}\big|^\alpha}=0.
	$$

\par If $p\to\infty$, then (cf. \cite[Theorem 1]{JXY})
$$
\begin{cases}
\frac{\alpha}{p}<\frac1p-\frac{1}{n}\to -\frac1n<0;\\
\big({\rm cap}_p\big(\overline{B(x,r)},\Omega\big)\big)^\frac1p\to {\rm cap}_\infty\big(\overline{B(x,r)},\Omega\big),
	\end{cases}
$$
and hence
\begin{align*}
\limsup_{p\to\infty}\left(\frac{{\rm cap}_p\big(\overline{B(x,r)},\Omega\big)}{|\overline{B(x,r)}|^{\alpha}{}}\right)^\frac1p
&=\limsup_{p\to\infty}\frac{
\big({\rm cap}_p\big(\overline{B(x,r)},\Omega\big)\big)^\frac1p}{|\overline{B(x,r)}|^{\frac{\alpha}{p}}}\\
&\le |\overline{B(x,r)}|^{\frac{1}{n}}
\big({\rm cap}_\infty\big(\overline{B(x,r)},\Omega\big)\big)\\
&\to 0\ \ \text{as}\ \ r\to 0.
\end{align*}

(ii) If $p=n,$ then
	$$\begin{aligned}
	\frac{\exp\left(-\beta ({\rm cap}_n\big(\overline{B(x,r)},\Omega\big))^{\frac{1}{1-n}}\right)}{|\overline{B(x,r)}|}&\leq \frac{\exp\left(-\beta ({\rm cap}_n(\overline{B(x,r)},B(x,R)))^{\frac{1}{1-n}}\right)}{|\overline{B(x,r)}|}\\
	& \leq \frac{\exp\left(-\beta \int_r^R S_t^{\frac{1}{1-n}}dt\right)}{|\overline{B(x,r)}|}
	\end{aligned}
	$$
	For
	$$
	\beta> n^{\frac{n}{n-1}}\nu_n^{\frac{1}{n-1}},
	$$
	there are sufficiently small $\varepsilon$ \& $R$ such that $$
	\frac{S_t}{t^{n-1}}<n\nu_n+\varepsilon<\Big(\frac{\beta}{n}\Big)^{n-1}\ \ \forall\ \ t<R.
	$$
	Accordingly, we get
	\begin{align*}
	\frac{\exp\left(-\beta \int_r^R S_t^{\frac{1}{1-n}}dt\right)}{|\overline{B(x,r)}|}&\leq\frac{\exp\left(-\beta (n\nu_n+\varepsilon)^{\frac{1}{1-n}}\int_r^R \frac{1}{t}dt\right)}{|\overline{B(x,r)}|}\\
	&=\Big(\frac{r}{R}\Big)^{\beta (n\nu_n+\varepsilon)^{\frac{1}{1-n}}}\frac{1}{|\overline{B(x,r)}|}.
	\end{align*}
	Notice that
	$$\beta (n\nu_n+\varepsilon)^{\frac{1}{1-n}}>n.
	$$
	So, via setting $r\rightarrow 0,$ we derive
	$$
	\lim\limits_{r\rightarrow 0}\frac{\exp\left(-\beta ({\rm cap}_n\big(\overline{B(x,r)},\Omega\big))^{\frac{1}{1-n}}\right)}{|\overline{B(x,r)}|}=0,
	$$
	as desired in \eqref{1.2}.
\end{proof}

\section{Nonvanishing relative $p$-capacities under $\alpha\ge 1-\frac{p}{n}$}\label{s2}

\begin{theorem}\label{t21} Given
	$$
	\alpha\ge 1-\frac{p}{n}\le 1-\frac1n,
	$$
	suppose that not only
  $$
  (\mathbb M^n,g)=\left([0,+\infty)\times \mathbb{S}^{n-1},dt^2+\varphi^2(t)g_{\mathbb S^{n-1}}\right)\ \ \text{with}\ \ \varphi''\geq 0.
  $$
  but also $\Omega$ obeys the isoperimetry of $o=\{t=0\}$-centered balls - namely - if the surface area $|\partial\tilde{K}|$ of any compact subset $\tilde{K}$ of $\Omega$ with the same volume as a geodesic ball $B(o,r)$  is not less than the surface area of $B(o,r)$ -i.e.-
  $$
  |\partial\tilde{K}|\geq |\partial B(o,r)|.
  $$
There hold four nonvanishing results for ${\rm cap}_p(K,\Omega)$.
\begin{itemize}
\item [\rm(i)] If $1\le p<n,$ then
\begin{equation}\label{2.1}
  \inf\limits_{K\subset\subset \Omega}\frac{{\rm cap}_p(K,\Omega)}{|K|^\alpha|\Omega|^{1-\frac{p}{n}-\alpha}}\geq\begin{cases}
  n\nu_n^{\frac{p}{n}}\left(\frac{n-p}{p-1}\right)^{p-1}&\ \ \text{as}\ \ p\not=1;\\
  n\nu_n^\frac1n&\ \ \text{as}\ \ p=1.
  \end{cases}
\end{equation}
with equality being approachable whenever $\alpha=1-\frac{p}{n}$.

\item[\rm(ii)] If $p=n,$ then
\begin{equation}\label{2.2}
\inf\limits_{K\subset\subset \Omega}\frac{\exp\left(-\beta \big({\rm cap}_n(K,\Omega)\big)^{\frac{1}{1-n}}\right)}{|K||\Omega|^{-1}}\geq 1\ \ \forall\ \ \beta\in (0,n^{\frac{n}{n-1}}\nu_n^{\frac{1}{n-1}}],
\end{equation}
with equality being approachable whenever $\beta=n^\frac{n}{n-1}\nu_n^{\frac{1}{n-1}}$.

\item[\rm(iii)] If $n<p<\infty$, then
\begin{equation}\label{2.3}
  \inf\limits_{K\subset\subset \Omega}\frac{{\rm cap}_p(K,\Omega)}{|\Omega|^{1-\frac{p}{n}}}\geq n\nu_n^{\frac{p}{n}}\left(\frac{p-n}{p-1}\right)^{p-1},
\end{equation}
with equality being approachable.

\item[\rm(iv)] If $p=\infty$, then
\begin{equation}\label{2.4}
\inf\limits_{K\subset\subset \Omega}\frac{{\rm cap}_\infty(K,\Omega)}{|\Omega|^{-\frac{1}{n}}}\geq \nu_n^{\frac{1}{n}},
\end{equation}
with equality being approachable.
\end{itemize}

\end{theorem}

\begin{proof} (i)
	We only need to consider
	$$
\eqref{2.1}\ \ \text{under}\ \ \alpha=1-\frac{p}{n}<1-\frac1n,
	$$
	thanks to
	$$\frac{{\rm cap}_p(K,\Omega)}{|K|^\alpha|\Omega|^{1-\frac{p}{n}-\alpha}}=\left(\frac{{\rm cap}_p(K,\Omega)}{|K|^{1-\frac{p}{n}}}\right)\left(\frac{|\Omega|^{\alpha-(1-\frac{p}{n})}}{|K|^{\alpha-(1-\frac{p}{n})}}\right)\geq\frac{{\rm cap}_p(K,\Omega)}{|K|^{1-\frac{p}{n}}}.
	$$
	For any compact set $K$ in $\Omega,$ assume that
	$$|K|=\big|\overline{B(o,r)}\big|\ \ \&\ \ \Omega=|B(o,R)|\ \ \text{for some $r<R$},
	$$
	where $B(o,r)$ \& $B(o,R)$ are geodesic balls centered at $o=\{t=0\}.$
	
	Since $(\mathbb M^n,g)$ satisfies the isoperimetry of $o=\{t=0\}$-centered balls, there is the isoperimetric function:
	$$
	\begin{cases}
	I(\tau)=\inf\Big\{|\partial K|:\ K\ \text{is precompact open in}\ \Omega\ \&\ |K|\geq\tau\Big\}=|\partial B(o,t)|;\\
	\text{$t$ is determined by $|B(o,t)|=\tau.$}
	\end{cases}
	$$
	Consequently, Maz’ya’s isocapacitary inequalities within either \cite[(2.2.8)]{Maz6} or \cite{grigor1999isoperimetric} indicates
	$$\begin{aligned}
	{\rm cap}_p(K,\Omega)&\geq\left(\int_{|K|}^{|\Omega|} \big(I(\tau)\big)^{\frac{p}{1-p}}d\tau\right)^{1-p}
	\\&=\left(\int_{r}^{R} \big(I(|B(o,t)|)\big)^{\frac{p}{1-p}}\left(\frac{d\tau}{dt}\right)\,dt\right)^{1-p}
	\\&=\left(\int_{r}^{R} |\partial B(o,t)|^{\frac{p}{1-p}}|\partial B(o,t)|\, dt\right)^{1-p}
	\\&={\rm cap}_p\big(\overline{B(o,r)},B(o,R)\big)
	\end{aligned}
	$$
This in turn implies
	$$\begin{aligned}
	\frac{{\rm cap}_p(K,\Omega)}{|K|^{\alpha}}&\geq\frac{{\rm cap}_p\big(\overline{B(o,r)},B(o,R)\big)}{|\overline{B(o,r)}|^{\alpha}}\\
	&=\frac{\left(\int_r^R S_t^{\frac{1}{1-p}}dt\right)^{1-p}}{|\overline{B(o,r)}|^{\alpha}}
	\\ &=\left(\frac{\int_r^R S_t^{\frac{1}{1-p}}dt}{\big(\int_0^r S_t\,dt\big)^{\alpha(1-p)^{-1}}}\right)^{1-p}\\
	&=
	\left(\frac{\int_r^R \big(n\nu_n\varphi^{n-1}(t)\big)^{\frac{1}{1-p}}dt}{(\int_0^r \big(n\nu_n\varphi^{n-1}(t)\big)\, dt)^{\alpha(1-p)^{-1}}}\right)^{1-p}
	\\ &=\big(F(r)\big)^{1-p}(n\nu_n)^{1-\alpha},
	\end{aligned}
	$$
	where
	$$
	F(r)=\left(\int_0^r\varphi^{n-1}(t)\, dt\right)^{\frac{\alpha}{p-1}}\int_r^R \varphi^{\frac{n-1}{1-p}}(t)\, dt.
	$$
\begin{itemize}
	\item On the one hand, upon noticing that
	$$\varphi(r)\sim r\ \ \text{as}\ \ r\rightarrow 0,
	$$
	we get
	$$\begin{aligned}
	F(0)&=\lim\limits_{r\rightarrow 0} F(r)\\
	&=\lim\limits_{r\rightarrow 0}\frac{-\varphi^{\frac{n-1}{1-p}}(r)}{\alpha(1-p)^{-1}\left(\int_0^r\varphi^{n-1}(t)\, dt\right)^{\alpha(1-p)^{-1}-1}\varphi^{n-1}(r)}
	\\ &=\lim\limits_{r\rightarrow 0}\frac{-r^{\frac{n-1}{1-p}}}{\alpha(1-p)^{-1}\left(\int_0^r\varphi^{n-1}(t)\, dt\right)^{\alpha(1-p)^{-1}-1}r^{n-1}}\\
	&=\alpha^{-1}({p-1})\lim\limits_{r\rightarrow 0}\frac{r^{\frac{(n-1)p}{1-p}}}{\left(\int_0^r\varphi^{n-1}(t)\, dt\right)^{\alpha(1-p)^{-1}-1}}
	\\ &=\Bigg(\frac{n(p-1)}{n-1}\Bigg)\left(\lim\limits_{r\rightarrow 0}\frac{r^n}{\int_0^r\varphi^{n-1}(t)\, dt}\right)^{\frac{(n-1)p}{n(1-p)}}\\
	&=\Bigg(\frac{n(p-1)}{n-1}\Bigg)n^{\frac{(n-1)p}{n(1-p)}}.
	\end{aligned}$$

	\item On the other hand, we can prove that $F(r)$ obtains its maximum at $r=0.$ In order to achieve this,
	we compute
	$$
	F'(r)=\left(\int_0^r\varphi^{n-1}(t)\, dt\right)^{\frac{\alpha}{p-1}-1}\big(\varphi(r)\big)^{n-1}G(r),
	$$
	where
	$$G(r)={\alpha}{(p-1)^{-1}}\int_r^R\varphi^{\frac{n-1}{1-p}}(t)\, dt-\varphi^{\frac{(n-1)p}{1-p}}(r)\int_0^r\varphi^{n-1}(t)\, dt
	$$
	satisfies
	$$
	\begin{cases}
	G(R)<0;\\
	G'(r)=\Big(\frac{(n-1)p}{p-1}\Big)\big(\varphi(r)\big)^{\frac{n-1}{1-p}-n}H(r);\\
	H(r)=\varphi'(r)\int_0^r\varphi^{n-1}(t)\, dt-{n}^{-1}\varphi^n(r).
	\end{cases}
	$$
	That
	$$
	H(0)=0\ \ \&\ \ H'(r)=\varphi''(r)\int_0^r\varphi^{n-1}(t)\, dt\geq 0
	$$
	would imply that
	$$H(r)\geq 0\ \ \&\ \ G(r)\leq G(R)<0.
	$$
	As a consequence, we achieve that not only $F'(r)\leq 0$ but also  $F(r)$ obtains its maximum at $r=0.$ Accordingly, there holds
	$$
	\inf\limits_{K\subset\subset\Omega}\frac{{\rm cap}_p(K,\Omega)}{|K|^{1-\frac{p}{n}}}\geq\inf\limits_{0<r<R}\big(F(r)\big)^{1-p}(n\nu_n)^{1-({1-\frac{p}{n}})}
	=\big({F}(0)\big)^{1-p}(n\nu_n)^{\frac{p}{n}}.
	$$
\end{itemize}

	In the end, from the above argument we know that the above-verified estimation is optimal in the sense that if
	$$
	(K,\Omega)=\big(\overline{B(o,r)},B(o,R)\big)\ \ \&\ \ r\rightarrow 0.
	$$
then \eqref{2.1}'s inequality approaches an equality whenever $\alpha=1-\frac{p}{n}$.

	(ii) Assume that
	$$|K|=\big|\overline{B(o,r)}\big|\ \ \&\ \ \Omega=|B(o,R)|\ \ \text{for some $r<R$}.
	$$
	Then
	$$\begin{aligned}
	\frac{\exp\left(-\beta\big({\rm cap}_n(K,\Omega)\big)^{\frac{1}{1-n}}\right)}{|K|}&\geq\frac{\exp\left(-n^{\frac{n}{n-1}}\nu_n^{\frac{1}{n-1}}\int_r^R \big(n\nu_n\varphi^{n-1}(t)\big)^{\frac{1}{1-n}}dt\right)}{\big|\overline{B(o,r)}\big|}
	\\& =\frac{\exp\Big(-n\int_r^R\frac{dt}{\varphi(t)}\Big)}{\int_0^rn\nu_n\varphi^{n-1}(t)\, dt}\\
	&=\frac{1}{n\nu_n\bar{F}(r)},
	\end{aligned}
	$$
	where
	$$
	\bar{F}(r)=\exp\left(n\int_r^R\frac{dt}{\varphi(t)}\right)\int_0^r\varphi^{n-1}(t)\, dt.
	$$
	
	We are going to show that $\bar{F}(r)$ is an increasing function for $r\in[0,R].$ Notice that not only
	$$
	\bar{F}'(r)=\exp\left(n\int_r^R\frac{dt}{\varphi(t)}\right)\frac{1}{\varphi(r)}\left(\varphi^n(r)-n\int_0^r\varphi^{n-1}(t)\, dt\right)
	$$
	but also
	$$
	\bar{F}'(r)\geq 0\Leftrightarrow\left(\varphi^n(r)-n\int_0^r\varphi^{n-1}(t)\, dt\right)\geq 0,
	$$
	whose right-hand inequality holds provided $\varphi'\geq 1$ which holds as $\varphi''\geq 0.$ Hence, we obtain
	$$
	\inf\limits_{K\subset\subset\Omega}|K|^{-1}{\exp\left(-\beta \big({\rm cap}_n(K,\Omega)\big)^{\frac{1}{1-n}}\right)}\geq
	\inf\limits_{0<r<R}\frac{1}{n\nu_n\bar{F}(r)}\geq \frac{1}{n\nu_n\bar{F}(R)}
	={|\Omega|}^{-1}.
	$$
	
	From the proof we know that the above-verified inequality is optimal in the sense that if
	$$
	\begin{cases}
\beta=n^{\frac{n}{n-1}}\nu_n^{\frac{1}{n-1}};\\
(K,\Omega)=\big(\overline{B(o,r)},B(o,R)\big);\\
r\rightarrow R,
\end{cases}
$$
then \eqref{2.2}'s inequality approaches an equality.

(iii) Assume that $$|K|=\big|\overline{B(o,r)}\big|\ \ \&\ \ \Omega=|B(o,R)|\ \ \text{for some $r<R.$}
$$
Then
	$$\begin{aligned}
	{\rm cap}_p(K,\Omega)|\Omega|^{\frac{p}{n}-1}&\geq \left(\int_r^R \big(n\nu_n\varphi^{n-1}(t)\big)^{\frac{1}{1-p}}dt\right)^{1-p}|\Omega|^{\frac{p}{n}-1}
	\\&\geq \left(\int_0^R \big(n\nu_n\varphi^{n-1}(t)\big)^{\frac{1}{1-p}}dt\right)^{1-p}\left(\int_0^R n\nu_n\varphi^{n-1}(t)\, dt\right)^{\frac{p}{n}-1}
	\\&=(n\nu_n)^{\frac{p}{n}} \big(\tilde{F}(R)\big)^{1-p},
	\end{aligned}
	$$
	where
	$$
	\tilde{F}(R)=\left(\int_0^R \varphi^{n-1}(t)\, dt\right)^{-\frac{p-n}{n(p-1)}}\int_0^R \varphi^{\frac{n-1}{1-p}}(t)\, dt.
	$$
	\begin{itemize}
		\item
	On the one hand, there holds
	$$\begin{aligned}
	\tilde{F}(0)&=\lim\limits_{R\rightarrow 0}\left(\frac{\varphi^{\frac{n-1}{1-p}}(R)}{\frac{p-n}{n(p-1)}\left(\int_0^R \varphi^{n-1}(t)\, dt\right)^{\frac{p-n}{n(p-1)}-1}\varphi^{n-1}(R)}\right)
	\\&=\left(\frac{n(p-1)}{p-n}\right)\lim\limits_{R\rightarrow 0}\left(\frac{R^{\frac{n-1}{1-p}}}{\left(\frac{R^n}{n}\right)^{\frac{p-n}{n(p-1)}-1}R^{n-1}}\right)\\
	&=\left(\frac{p-1}{p-n}\right)n^{\frac{p-n}{n(p-1)}}.
	\end{aligned}
	$$

\item	On the other hand, by a similar analysis to (i), we can get that $\tilde{F}'\leq 0$ provided $\varphi''\geq 0.$ Then
	$$
	\inf\limits_{K\subset\subset\Omega}{\rm cap}_p(K,\Omega)|\Omega|^{\frac{p}{n}-1}\geq(n\nu_n)^{\frac{p}{n}} \big(\tilde{F}(0)\big)^{1-p}=
	n\nu_n^{\frac{p}{n}}\left(\frac{p-n}{p-1}\right)^{p-1}.
	$$
\end{itemize}

	Of course, \eqref{2.3}'s equality is approachable in the sense that if
	$$(K,\Omega)=\big(\overline{B(o,r)},B(o,R)\big)\ \ \&\ \ r<R\rightarrow 0,
	$$
	then \eqref{2.3}'s inequality tends to an equality.

	(iv) \eqref{2.4} follows from \cite[Theorem 1]{JXY} \& \eqref{2.3} via
	$$
	\inf\limits_{K\subset\subset \Omega}\left(\frac{{\rm cap}_p(K,\Omega)}{|\Omega|^{1-\frac{p}{n}}}\right)^\frac1p\geq n^\frac1p\nu_n^{\frac{1}{n}}\left(\frac{p-n}{p-1}\right)^\frac{p-1}{p}\to \nu_n^\frac1n\ \ \text{as}\ \ p\to\infty,
	$$
	
	Similary, \eqref{2.4}'s equality is approachable in the sense that if
	$$(K,\Omega)=\big(\overline{B(o,r)},B(o,R)\big)\ \ \&\ \ r<R\rightarrow 0,
	$$
	then \eqref{2.4}'s inequality tends to an equality.
	
\end{proof}

\begin{remark} Below are six comments.
\begin{itemize}

\item Under the hypothesis that $\Omega$ enjoys the isoperimetry of $o=\{t=0\}$-centered balls, we can obtain the following stronger isocapacitary $(1,\infty)\ni p$-inequality in a similar way to establish \cite[Theorem 2.1]{JX}:
$$
 \exists\ \tau_0\in (0,1)\ \text{such that}
 \ \left(\frac{{\rm cap}_{p}(K,\Omega)}{{\rm cap}_{p}\big(\overline{B(o,r)},B(o,R)\big)}\right)^{\frac{1}{p}}\geq \frac{|\partial D_{\tau_0}|}{|\partial B(o,t_0)|}\geq 1,
$$
where
	$$
	\begin{cases}
	D_\tau=\{x\in \Omega\setminus K:\ u(x)\ge \tau\}\cup K;\\
 u\ \text{is the $p$-capacitary potential of $(K,\Omega)$};\\
 t_0 \ \text{is determined by}\ |B(o,t_0)|=|D_{\tau_0}|.
    \end{cases}
    $$

\item In \eqref{2.2}, the condition $``\varphi''\geq 0"$ can be weakened to $``\varphi'\geq 1."$

\item From the geometric point of view, if we let not only $\mathrm{Sec}[g]$ but also $\mathrm{Ric}[g]$ denote not only the sectional curvature of $(\mathbb M^n,g)$ but also the Ricci curvature of $(\mathbb M^n,g),$ then
  $$\varphi''\geq 0\Leftrightarrow \mathrm{Sec}[g]\leq 0\Leftrightarrow {\mathrm{Ric}}[g]\leq 0.$$

\item If
$$
\varphi(t)=t\ \ \text{or}\ \  \varphi(t)=\sinh t,
$$
then $\mathbb M^n$ is the Euclidean space $\mathbb R^n$ or the hyperbolic space $\mathbb H^{n}$. Thus, Theorem \ref{t21} holds for these two special manifolds; see also \cite{AX, JX} for certain of the related results.

\item
According to \cite[Lemma 2.1]{GU} (cf. \cite{K}), if
$$
\begin{cases}
\mathbb M^n=\mathbb R^n;\\n-1<p \leq n;\\
\mathrm{diam}(K)=\text{the diameter of $K$},
\end{cases}
$$
then there is a constant $c_{p,n}>0$ depending on $\{p,n\}$ such that
$$
\frac{{\rm cap}_p(K,\Omega)}{(\mathrm{diam}(K))^\frac{p}{n-1}|\Omega|^\frac{{ n-1-p}}{n-1}}\ge c_{p,n}.
$$
Obviously, this last inequality can be treated as a rough variant of the special case $\alpha=\frac{p}{n(n-1)}$ of \eqref{2.1} {with $p\in(n-1,n).$ While for $p=n,$ \eqref{2.2} is stronger than the last rough inequality.}

\item Theorem \ref{t21} is a local version of the capacity-volume estimates. So, it is natural to consider the global version by setting not only
  $$
  \Omega\rightarrow \mathbb M^n=[0,+\infty)\times \mathbb{S}^{n-1},
  $$
but also
 $$
  \begin{cases}
  \widetilde{{\rm cap}}_p(K)=\inf\limits_{\Omega\subseteq\mathbb M^n}{\rm cap}_p(K,\Omega)&\ \  \text{as}\ \  p\in [1,n);\\
  \widetilde{{\rm cap}}_n(K)=\inf\limits_{\Omega\subseteq\mathbb M^n}|\Omega|\exp\left(-n^{\frac{n}{n-1}}\nu_n^{\frac{1}{n-1}} \big({\rm cap}_n(K,\Omega)\big)^{\frac{1}{1-n}}\right)&\ \ \text{as}\ \ p=n;\\
  \widetilde{{\rm cap}}_p(K)=\inf\limits_{\Omega\subseteq\mathbb M^n}|\Omega|^{\frac{p}{n}-1}{\rm cap}_p(K,\Omega)&\ \ \text{as}\ \ p\in (n,\infty);\\
  \widetilde{{\rm cap}}_\infty(K)=\inf\limits_{\Omega\subseteq\mathbb M^n}|\Omega|^{\frac{1}{n}}{\rm cap}_\infty(K,\Omega)&\ \ \text{as}\ \ p=\infty.
  \end{cases}
  $$
  Accordingly, we have the global $p$-isocapacitary inequalities below:
  	\begin{equation*}
  	\begin{cases}
  	\inf\limits_{K\subset\subset\mathbb M^n}\frac{\widetilde{\rm cap}_p(K)}{|K|^{1-\frac{p}{n}}}\geq
  	n\nu_n^{\frac{p}{n}}\left(\frac{n-p}{p-1}\right)^{p-1}&\ \ \text{as}\ \  p\in [1,n);\\
  	\inf\limits_{K\subset\subset\mathbb M^n}\frac{\widetilde{{\rm cap}}_n(K)}{|K|}\geq 1&\ \ \text{as}\ \  p=n;\\
  	\inf\limits_{K\subset\subset\mathbb M^n}\widetilde{{\rm cap}}_p(K)\geq n\nu_n^{\frac{p}{n}}\left(\frac{p-n}{p-1}\right)^{p-1}&\ \ \text{as}\ \  p\in (n,\infty);\\
  	\inf\limits_{K\subset\subset\mathbb M^n}\widetilde{{\rm cap}}_\infty(K)\geq \nu_n^{\frac{1}{n}}&\ \ \text{as}\ \  p=\infty;
  	\end{cases}
  	\end{equation*}
  thereby obtaining the generalized isoperimetric inequality (cf. \cite[Lemma 2.2.5]{Maz6})
  $${\widetilde{\rm cap}}_1(K)=\inf\Big\{|\Sigma|:\ \text{all admissible surfaces $\Sigma$ enclosing $\partial K$}\Big\}\geq n\nu_n^{\frac{1}{n}}|K|^{1-\frac{1}{n}}.
  $$
\end{itemize}
  \end{remark}

\section{Applying $\alpha=1-\frac{p}{n}$ to sharp weak $(p,q)$-imbeddings}\label{s3}
\begin{theorem}\label{t31}
	Let
	$$
	\begin{cases}
	(\mathbb M^n,g)=\left([0,+\infty)\times \mathbb{S}^{n-1},dt^2+\varphi^2(t)g_{\mathbb S^{n-1}}\right)\ \ \text{with}\ \ \varphi''\geq 0;\\
	\text{$\Omega=$ a bounded domain in $\mathbb M^n$ with the isoperimetry of $o=\{t=0\}$-centered balls};\\
	c_{n,p}=
     \begin{cases}
            n^{-1}\nu_n^{-\frac{1}{n}} \ &\ \text{as} \ p=1;\\
            n^{-\frac{1}{p}}\nu_n^{-\frac{1}{n}}\left(\frac{n-p}{p-1}\right)^{\frac{1-p}{p}} \ &\ \text{as}\ p\in(1,n);\\
            n^{-1}\nu_n^{-\frac{1}{n}} \ &\ \text{as} \ p=n;\\
            n^{-\frac{1}{p}}\nu_n^{-\frac{1}{n}}\left(\frac{p-n}{p-1}\right)^{\frac{1-p}{p}} \ &\ \text{as}\ p\in(n,\infty);\\
            \nu_n^{-\frac{1}{n}} \ &\ \text{as}\ p=\infty.
      \end{cases}
            \end{cases}
	$$
Then there hold four sharp weak $(p,q)$-imbeddings.
\begin{itemize}
    \item [\rm(i)] If $1\le p<n,$ then
        \begin{equation}\label{3.1}
        \|u\|_{\frac{np}{n-p},\infty}\leq c_{n,p}\|\nabla u\|_p
           \ \ \ \forall \ u \in W^{1,p}_0(\mathbb M^n).
        \end{equation}
        with equality being approachable.

    \item[\rm(ii)] If $p=n,$ then
        \begin{equation}\label{3.2}
        \sup\limits_{\lambda>0}\lambda\left(\ln\bigg(\frac{|\Omega|}{\big|\{x\in\Omega:\ |u(x)|\geq\lambda\}\big|}\bigg)\right)^{\frac{1}{n}-1}\leq c_{n,n}\|\nabla u\|_n
           \ \ \ \forall \ u \in W^{1,n}_0(\Omega).
        \end{equation}
        with equality being approachable.

    \item[\rm(iii)] If $p\in(n,\infty),$ then
        \begin{equation}\label{3.3}
    \|u\|_\infty\leq c_{n,p}|\Omega|^{\frac{p-n}{np}}\|\nabla u\|_p\ \ \ \ \forall  \ u \in W^{1,p}_0(\Omega).
        \end{equation}with equality being approachable.
    \item[\rm(iv)] If $p=\infty$, then
     \begin{equation}\label{3.4}
    \|u\|_\infty\leq c_{n,\infty}|\Omega|^{\frac{1}{n}}\|\nabla u\|_\infty\ \ \ \ \forall  \ u \in W^{1,\infty}_0(\Omega).
        \end{equation}with equality being approachable.
\end{itemize}
\end{theorem}

\begin{proof}
    In the sequel, we always assume that not only
    $$
\begin{cases}
u \in W^{1,p}_0(\mathbb M^n)&\text{as}\ \ p\in [1,n)\ \text{via letting}\ \ \Omega\to\mathbb M^n;\\
u \in W^{1,p}_0(\Omega)&\text{as}\ \ p\in [n,\infty],
\end{cases}
    $$
    but also
    $$
    [u]_\lambda=\big\{x\in\mathbb M^n\ \text{or}\ \Omega:\ |u(x)|\geq \lambda\big\}\ \ \text{with}\ \lambda^{-1}|u(\cdot)|\geq \mathrm{1}_{[u]_\lambda}\ \ \forall\ \ \lambda\in (0,\infty).
    $$

    (i) Via choosing $\alpha=1-\frac{p}{n}$ within \eqref{2.1}, we can obtain
    $$\begin{aligned}
   c_{n,p}^{-p}\big|[u]_\lambda\big|^{1-\frac{p}{n}}&\leq\inf\limits_{\Omega\subseteq\mathbb M^n}{\rm cap}_p([u_\lambda],\Omega)
    \\&\leq\int_{\mathbb M^n}|\nabla (\lambda^{-1}|u|)|^p d\upsilon_g
    \\&=\lambda^{-p}\int_{\mathbb M^n}|\nabla u|^p d\upsilon_g
    \end{aligned}
    $$
    Hence \eqref{3.1} follows from
    $$
    \lambda\big|[u]_\lambda\big|^{\frac{n-p}{np}}\leq c_{n,p}\|\nabla u\|_p\ \ \forall\ \  \lambda>0.
    $$

    However, regarding the \eqref{3.1}'s sharpness in the sense that $c_{n,p}$ is approachable, we deal with two cases.

    \begin{itemize}
    	\item If $1<p<n$, then we construct the $(1,n)\ni p$-capacitary potential of a geodesic ball $\overline{B(o,r)}$ according to the formula
    $$
    u_{p,r}(x)=\begin{cases}
        1 &\ \text{as}\ \ x\in \overline{B(o,r)};\\
        1-\frac{\int_r^{d(o,x)}\varphi^{\frac{n-1}{1-p}}(s)\, ds}{\int_r^{\infty}\varphi^{\frac{n-1}{1-p}}(s)\, ds} & \ \ \text{as}\ x\in\mathbb{M}^n\setminus\overline{B(o,r)},
    \end{cases}
    $$
    where not only $d(o,\cdot)$ denotes the distance function of $o$ but also the infinite integral is convergent due to $\varphi''\geq 0,$ thereby getting that not only
    $$
    \|u_{p,r}\|_{\frac{np}{n-p},\infty}\geq \big|[u_{p,r}]_1\big|^{\frac{1}{p}-\frac{1}{n}}=\big|\overline{B(o,r)}\big|^{\frac{1}{p}-\frac{1}{n}}
    $$
    but also
    \begin{align}\label{3.5}
    \|\nabla u_{p,r}\|_p^p&=\left(\int_r^{\infty}\varphi^{\frac{n-1}{1-p}}(s)\, ds\right)^{-p}\int_{\mathbb{M}^n}|\varphi^{\frac{n-1}{1-p}}(d(o,x))\nabla d(o,x)|^p\, d\upsilon_g(x)
    \\
    &=\left(\int_r^{\infty}\varphi^{\frac{n-1}{1-p}}(s)\, ds\right)^{-p}\int_r^\infty|\varphi^{\frac{n-1}{1-p}}(s)|^p n\nu_n\varphi^{n-1}(s)\, ds
    \notag\\
    &= \widetilde{{\rm cap}}_p\big(\overline{B(o,r)}\big).\notag
    \end{align}
    Consequently, we achieve
    $$
    \frac{\|u_{p,r}\|_{\frac{np}{n-p},\infty}}{\|\nabla u_{p,r}\|_p}\geq\frac{\big|\overline{B(o,r)}\big|^{\frac{n-p}{np}}}{\Big(\widetilde{{\rm cap}}_p\big(\overline{B(o,r)}\big)\Big)^{\frac{1}{p}}}\rightarrow \left( n\nu_n^{\frac{p}{n}}\left(\frac{n-p}{p-1}\right)^{p-1}\right)^{-\frac{1}{p}}=c_{n,p}\ \ \text{as}\ \ r\rightarrow0
    $$
    thereby finding that \eqref{3.1}'s equality for $p\in (1,n)$ is approachable.

    \item If $p=1$, then we choose
      $$
    u_{1,r}(x)=\begin{cases}
        1 &\ \text{as} \ x\in \overline{B(o,r)};\\
        1+r^{-2}\big(r-d(o,x)\big) & \ \text{as}\ x\in  \overline{B(o,r+r^2)}\setminus \overline{B(o,r)};\\
        0 & \ \text{as}\ x\in\mathbb{M}^n\setminus \overline{B(o,r+r^2)},
    \end{cases}
    $$
    to calculate
    $$\begin{aligned}
    \liminf\limits_{r\rightarrow 0}\frac{\|u_{1,r}\|_{\frac{n}{n-1},\infty}}{\|\nabla u_{1,r}\|_1}&\geq \lim\limits_{r\rightarrow 0} \frac{|[u_{1,r}]_1|^{1-\frac{1}{n}}}{\int_{\mathbb{M}^n}|\nabla u_{1,r}|\, d\upsilon_g}
    \\ &=\lim\limits_{r\rightarrow 0}\frac{\big|\overline{B(o,r)}\big|^{1-\frac{1}{n}}}{\int_r^{r+r^2}r^{-2} n\nu_n\varphi^{n-1}(t)dt}
    \\ &=\lim\limits_{r\rightarrow 0}\frac{(\nu_n r^n)^{1-\frac{1}{n}}}{\int_r^{r+r^2}r^{-2} n\nu_n t^{n-1}dt}
    \\&=n^{-1}\nu_n^{-\frac{1}{n}}.
    \end{aligned}
    $$
    This just says that \eqref{3.1}'s equality for $p=1$ is approachable.
\end{itemize}

    (ii) According to \eqref{2.2}, for any $\lambda>0$ there holds
    $$\begin{aligned}
    \frac{\big|[u]_\lambda\big|}{|\Omega|}&\leq \exp\left(-n^{\frac{n}{n-1}}\nu_n^{\frac{1}{n-1}}({\rm cap}_n([u_\lambda],\Omega))^{\frac{1}{1-n}}\right)
    \\ &\leq\exp\left(-n^{\frac{n}{n-1}}\nu_n^{\frac{1}{n-1}}\bigg(\int_\Omega|\nabla(\lambda^{-1}|u|)|^nd\upsilon_g\bigg)^{\frac{1}{1-n}}\right)
    \\ &=\exp\left(-n^{\frac{n}{n-1}}\nu_n^{\frac{1}{n-1}}\lambda^{-\frac{n}{1-n}}\bigg(\int_\Omega|\nabla u|^nd\upsilon_g\bigg)^{\frac{1}{1-n}}\right),
    \end{aligned}
    $$
    which is actually equivalent to \eqref{3.2}.

    In order to study the sharpness of \eqref{3.2}, given a geodesic open ball $\Omega=B(o,R)$, we consider the $n$-capacitary potential $u_{n,r}$ of $\big(\overline{B(o,r)},B(o,R)\big)$ for some $r\in(0,R)$ given as below:
        $$
    u_{n,r}(x)=\begin{cases}
        1 &\ \text{as} \ \ x\in \overline{B(o,r)};\\
        1-\frac{\int_r^{d(o,x)}\varphi^{-1}(s)\, ds}{\int_r^{R}\varphi^{-1}(s)\, ds} & \ \text{as}\ \ x\in B(o,R)\setminus \overline{B(o,r)}.
    \end{cases}
    $$
    In a similar way to establish \eqref{3.5}, we have
    $$
    \|\nabla u_{n,r}\|_n^n={\rm cap}_p\big(\overline{B(o,r)},B(o,R)\big)=n\nu_n\left(\int_r^R\varphi^{-1}(s)\, ds\right)^{1-n},
    $$
    whence estimating
    $$\begin{aligned}
    \frac{ \sup\limits_{\lambda>0}\lambda\left(\ln\frac{|B(o,R)|}{\big|[u_{n,r}]_\lambda\big|}\right)^{\frac{1}{n}-1}}{\|\nabla u_{n,r}\|_n}&\geq
    \frac{\left(\ln|B(o,R)|-\ln\big|\overline{B(o,r)}\big|\right)^{\frac{1}{n}-1}}{\left({\rm cap}_p\big(\overline{B(o,r)},B(o,R)\big)\right)^{\frac{1}{n}}}
    \\ &=(n\nu_n)^{-\frac{1}{n}}\left(\frac{\ln|B(o,R)|-\ln\big|\overline{B(o,r)}\big|}{\int_r^R\varphi^{-1}(s)\, ds}\right)^{\frac{1}{n}-1}
    \\ &\rightarrow (n\nu_n)^{-\frac{1}{n}}\left(\frac{n\nu_n\varphi^n(R)}{|B(o,R)|}\right)^{\frac{1}{n}-1} \ \text{as}\  r\rightarrow R
    \\ &\rightarrow n^{-1}\nu_n^{-\frac{1}{n}}\ \text{as}\ R\rightarrow 0.
    \end{aligned}
    $$
    Of course, this tells that \eqref{3.2}'s equality is approachable.

    (iii) In accordance with \eqref{2.3}, for
    $$
    \begin{cases} 0<\lambda<\|u\|_\infty;\\ [u]_\lambda\neq\emptyset;\\
     \big|[u]_\lambda\big|>0,
     \end{cases}
     $$
     we can evaluate
    $$\begin{aligned}
   c_{n,p}^{-p}|\Omega|^{1-\frac{p}{n}}&\leq {\rm cap}_p([u]_\lambda,\Omega)
    \\&\leq\int_{\Omega}|\nabla (\lambda^{-1}|u|)|^p d\upsilon_g
    \\&=\lambda^{-p}\int_{\Omega}|\nabla u|^p d\upsilon_g.
    \end{aligned}
    $$
    So, sending $\lambda$ to $\|u\|_\infty$ in the above estimation gives
    $$\|u\|_\infty\leq c_{n,p}|\Omega|^{\frac{p-n}{np}}\|\nabla u\|_p,
    $$
    as desired in \eqref{3.3}.

    In order to reach \eqref{3.3}'s equality, let
    $$
    \begin{cases}
    \Omega=B(o,R);\\
    u_{p,r}=\text{the $p$-capacitary potential of}\ \big(\{o\},B(o,R)\big) - \text{i.e.}-\\
    u_{p,r}(x)=
        1-\frac{\int_0^{d(o,x)}\varphi^{\frac{n-1}{1-p}}(s)\, ds}{\int_0^{R}\varphi^{\frac{n-1}{1-p}}(s)\, ds} \ \ \forall\ \  x\in B(o,R).
        \end{cases}
    $$
    Then
    $$
    \begin{cases}
   \|u\|_\infty=1;\\
   \|\nabla u\|_p^p=n\nu_n\left(\int_0^R\varphi^{\frac{n-1}{1-p}}(s)\, ds\right)^{1-p}.
  \end{cases}
    $$
    Needless to say, the last two formulas derive
    $$\begin{aligned}
        |\Omega|^{\frac{p-n}{np}}\|\nabla u\|_p&=|B(o,R)|^{\frac{p-n}{np}}(n\nu_n)^{\frac{1}{p}}\left(\int_0^R\varphi^{\frac{n-1}{1-p}}(s)\, ds\right)^{\frac{1-p}{p}}
        \\ &\rightarrow (\nu_nR^n)^{\frac{p-n}{np}}(n\nu_n)^{\frac{1}{p}}\left(({p-1})({p-n})^{-1}R^\frac{p-n}{p-1}\right)^{\frac{1-p}{p}}
        \\ &=c_{n,p}^{-1}\ \ \text{as}\ \ R\rightarrow 0.
    \end{aligned}
    $$
    Therefore, \eqref{3.3}'s equality is approachable.

    (iv) Upon taking into account of \eqref{2.4}, we see that the proof is analogous to that of verifying (iii) so it is omitted here.
\end{proof}
\begin{remark}
By tracing the proof process of Theorem 3.1, we find that if 
$$
\begin{cases}
\varphi(t)=t;\\
\mathbb{M}^n=\mathbb{R}^n;\\
1<p\le\infty;\\
0<r<R<\infty,
\end{cases}
$$ 
the equalities within \eqref{3.1}-\eqref{3.2}-\eqref{3.3}-\eqref{3.4} can be achieved via the next functions
$$u_p(x)=
\begin{cases}
\left(\frac{\max\{|x|,r\}}{r}\right)^{\frac{p-n}{p-1}}\ \ \forall\ \ x\in \mathbb{R}^n&\ \ \text{as}\ \ 1<p<n;\\
 \frac{\ln R-\ln \max\{|x|,r\}}{\ln R-\ln r}\ \ \forall\ \ x\in B(o,R)&\ \ \text{as}\ \ p=n;\\
1-\left(\frac{|x|}{R}\right)^{\frac{p-n}{p-1}}\ \ \forall\ \ x\in B(o,R)&\ \ \text{as}\ \ n<p<\infty;\\
1-\frac{|x|}{R}\ \ \forall\ \ x\in B(o,R)&\  \ \text{as}\ p=\infty.
\end{cases}
$$
For certain of the closely-related results over $\mathbb R^n$, see also \cite{CRT2, CRT1, AX, X2} .
\end{remark}
\section{Applying $\alpha=1>1-\frac{p}{n}$ to principal $p$-frequencies}\label{s4}

\begin{theorem}\label{t41}
	 Let
	$$
	\begin{cases}
(\mathbb M^n,g)=\left([0,+\infty)\times \mathbb{S}^{n-1},dt^2+\varphi^2(t)g_{\mathbb S^{n-1}}\right)\ \ \text{with}\ \ \varphi''\geq 0;\\
\text{$\Omega$ obey the isoperimetry of $o=\{t=0\}$-centered balls.}
\end{cases}
	$$
	Then
	\begin{equation}
	\label{4.1}
\lambda_{1,p}(\Omega)
\geq \begin{cases} n\nu_n^{\frac{1}{n}}|\Omega|^{-\frac1n}&\ \ \text{as}\ \ p=1;\\
n\nu_n^{\frac{p}{n}}p^{-p}(\max\{n,p-n\})^{p-1}|\Omega|^{-\frac {p}{n}} &\ \ \text{as}\ \ p\in (1,\infty);\\
\nu_n^\frac1n|\Omega|^{-\frac1n}&\ \ \text{as}\ \ p=\infty.
\end{cases}
\end{equation}
In particular, \eqref{4.1}'s case at $p=1\ \text{or}\ p=\infty$ is sharp in the sense that
$$\mathbb M^n=\mathbb R^n\ \ \&\ \ \Omega=\mathbb B^n\Rightarrow
\lambda_{1,1}(\Omega)=n\ \ \&\ \ \lambda_{1,\infty}(\Omega)=1.
$$
\end{theorem}

\begin{proof} The sharpness of \eqref{4.1} at $p\in\{1,\infty\}$ follows from \cite[Proposition 3.1 \& Remark 3.4]{X}. Thus, it remains to verify \eqref{4.1}.
	
First of all, recall such an equivalent property of the Maz'ya constant $\gamma_p(\Omega)$ that (cf. \cite{maz2003lectures, grigor1999isoperimetric})
\begin{align}\label{4.2}
\lambda_{1,p}(\Omega)&\leq \gamma_p(\Omega)\\
&\equiv\inf\limits_{K\subset\subset\Omega}|K|^{-1}{{\rm cap}_p(K,\Omega)}\notag\\
&\leq\begin{cases}\lambda_{1,1}(\Omega)&\ \ \text{as}\ \ p=1;\\ {p^p}{(p-1)^{1-p}}\lambda_{1,p}(\Omega)&\ \ \text{as}\ \ p\in (1,\infty).\notag
\end{cases}
\end{align}

Next is a three-fold consideration.

\begin{itemize}
	\item If $p\in [1,n)$, then using Theorem \ref{t21}'s case $\alpha=1$, along with \eqref{4.2}, derives
\begin{equation}
\label{4.3}
\lambda_{1,p}(\Omega)\geq\begin{cases}
 n\nu_n^{\frac{1}{n}}|\Omega|^{-\frac{1}{n}}&\ \text{as}\ \ p=1;\\
{(p-1)^{p-1}}{p^{-p}}\gamma_p(\Omega)\geq
  n\nu_n^{\frac{p}{n}}{(n-p)^{p-1}}{p^{-p}}|\Omega|^{-\frac{p}{n}}&\ \text{as}\ \ p\in(1,n).
   \end{cases}
\end{equation}

\item However, thanks to \cite[Theorem 1.2]{jin2024lower}, we know that $$
[1,\infty)\ni p\mapsto p\big(\lambda_{1,p}(\Omega)\big)^\frac{1}{p}
$$
is an increasing function, whence getting
\begin{equation}
\label{4.4}
\lambda_{1,p}(\Omega)\geq\left(p^{-1}{\lambda_{1,1}(\Omega)}\right)^p
\geq n^p\nu_n^{\frac{p}{n}}{p^{-p}}|\Omega|^{-\frac{p}{n}}\ \ \forall\ \ p\in [1,\infty).
\end{equation}
It is easy to see that \eqref{4.4} not only is stronger than \eqref{4.3}'s case $p\in(1,n)$ but also ensures \eqref{4.1} when $p\in (1,2n]$.

\item If $n<p<\infty$, then \eqref{2.3} tells us that not only
$$
\gamma_p(\Omega)\geq \inf\limits_{K\subset\subset\Omega}|\Omega|^{-1}{{\rm cap}_p(K,\Omega)}\geq n\nu_n^{\frac{p}{n}}\left(\frac{p-n}{p-1}\right)^{p-1}|\Omega|^{-\frac{p}{n}},
$$
but also
\begin{equation}
\label{4.5}
\lambda_{1,p}(\Omega)\geq{(p-1)^{p-1}}{p^{-p}}\gamma_p(\Omega)\geq n\nu_n^{\frac{p}{n}}{(p-n)^{p-1}}{p^{-p}}|\Omega|^{-\frac{p}{n}},
\end{equation}
which ensures \eqref{4.1}'s case $p\in [2n,\infty)$.
\end{itemize}

Finally, a combination of \eqref{4.5} \& the limiting formula (cf. \cite[Remark 3.4]{X} or \cite{JLM, KF} for $\mathbb M^n=\mathbb R^n$)
$$
\lambda_{1,\infty}(\Omega)=\lim_{p\to\infty}\big(\lambda_{1,p}(\Omega)\big)^\frac1p=\lim_{p\to\infty}\big(\gamma_{p}(\Omega)\big)^\frac1p
$$
derives
$$
\lambda_{1,\infty}(\Omega)\ge \lim_{p\to\infty}n^\frac1p\nu_n^{\frac{1}{n}}{(1-p^{-1}n)}(p-n)^{-\frac{1}{p}}|\Omega|^{-\frac{1}{n}}=\nu_n^\frac1n|\Omega|^{-\frac1n},
$$
as desired within \eqref{4.1}'s case $p=\infty$.

\end{proof}

\end{document}